\documentclass{amsart}

\usepackage{amsfonts, amsmath, amssymb, amsthm}
\usepackage{mathrsfs}
\usepackage{xypic}

\newcommand{\trace}{\mathrm{trace}}
\newcommand{\vol}{\mathrm{vol}}

\newcommand{\st}{\,\big|\,}
\newcommand{\real}{\mathbb{R}}
\newcommand{\nat}{\mathbb{N}}
\newcommand{\tor}{\mathbb{T}}
\newcommand{\integer}{\mathbb{Z}}
\newcommand{\complex}{\mathbb{C}}

\newtheorem{theorem}{Theorem}[section]
\newtheorem{lemma}[theorem]{Lemma}

\newtheorem{remark}[theorem]{Remark}

\numberwithin{equation}{section}

\begin{document}

\title[ENERGY-MINIMIZING MAPS FROM NONNEGATIVE RICCI CURVATURE]{ENERGY-MINIMIZING MAPS FROM MANIFOLDS WITH NONNEGATIVE RICCI CURVATURE}

\author[JAMES DIBBLE]{JAMES DIBBLE}
\address{Department of Mathematics, University of Iowa, 14 MacLean Hall, Iowa City, IA 52242}
\email{james-dibble@uiowa.edu}
\thanks{Most of these results appear in my doctoral dissertation \cite{Dibble2014}. I'm thankful to my dissertation advisor, Xiaochun Rong, as well as Christopher Croke, Penny Smith, and Charles Frohman for many helpful discussions. An anonymous referee pointed out a mistake that I made while reformulating some of the results in \cite{Dibble2014}. Parts of this work were done while I visited Capital Normal University.}


\date{}

\begin{abstract}
The energy of any $C^1$ representative of a homotopy class of maps from a compact and connected Riemannian manifold with nonnegative Ricci curvature into a complete Riemannian manifold with no conjugate points is bounded below by a constant determined by the asymptotic geometry of the target, with equality if and only if the original map is totally geodesic. This conclusion also holds under the weaker assumption that the domain is finitely covered by a diffeomorphic product, and its universal covering space splits isometrically as a product with a flat factor, in a commutative diagram that follows from the Cheeger--Gromoll splitting theorem.
\end{abstract}

\maketitle

\section{Introduction}

Eells--Sampson \cite{EellsSampson1964} showed that a $C^2$ map from a compact manifold with nonnegative Ricci curvature into a complete manifold with nonpositive sectional curvature is harmonic if and only if it is totally geodesic. Hartman \cite{Hartman1967} further showed that every such harmonic map minimizes energy within its homotopy class. This paper presents a partial generalization of these results under the additional assumption that the original map is homotopic to a totally geodesic map. This is done by showing that the asymptotic geometry of $N$ yields a lower bound for the energy of maps in a given homotopy class that is realized by, and only by, totally geodesic maps. In the special case of a domain with nonnegative Ricci curvature and a target with no conjugate points, this takes the following form.

\begin{theorem}\label{nnrc theorem}
    Let $M$ be a compact and connected $C^2$ Riemannian manifold with nonnegative Ricci curvature, $N$ a complete Riemannian manifold with no conjugate points, and $[F]$ a homotopy class of maps from $M$ to $N$. Then, for the flat semi-Finsler manifold $K$ and totally geodesic surjection $S : M \to K$ constructed in subsection 3.5, the following holds: For any $C^1$ map $f \in [F]$, $E(f) \geq E(S)$, with equality if and only if $f$ is totally geodesic.
\end{theorem}

\noindent The Cheeger--Gromoll splitting theorem \cite{CheegerGromoll1971} states that the Riemannian universal covering space of $M$ splits isometrically as a product $M_0 \times \real^k$, while the semi-Finsler universal covering space of $K$ is $\real^m$. The above map $S$ lifts to a totally geodesic map $M_0 \times \real^k \to \real^m$ that is constant on each $M_0$-fiber and an affine surjection on each $\real^k$-fiber. Moreover, if $f$ is totally geodesic, then $K$ is Riemannian and embeds isometrically in $N$ in such a way that $f = S$.

The more general version of Theorem \ref{nnrc theorem} holds for compact domains that are finitely covered by a product $M_1 \times \tor^k$ in a commutative diagram of the form \eqref{diagram1} that, by the Cheeger--Gromoll splitting theorem, holds whenever $M$ has nonnegative Ricci curvature. An example is given to show that such domains may have some negative Ricci curvature.

These results build on the work of Croke--Fathi \cite{CrokeFathi1990} relating energy and intersection. Without curvature assumptions on $M$ and $N$, they proved a lower bound for the energy of any $C^1$ representative of a homotopy class of maps $[F]$ from $M$ to $N$, one which is realized only by maps called homotheties. They define the \textbf{intersection} of a map $f : M \to N$ to be
\[
    i(f) = \lim_{t \to \infty} \frac{1}{t} \int_{SM} \phi_t(v) \,d\mathrm{Liou}_{SM}(v)\textrm{,}
\]
where $\mathrm{Liou}_{SM}$ is the Liouville measure on the unit sphere bundle $\pi : SM \to M$, $\Phi : SM \times \real \to SM$ is the geodesic flow, $\Phi_t(\cdot) = \Phi(\cdot,t)$, and
\begin{align*}
    \phi_t(v) = \min \{ L(\gamma) \st \gamma : [0,t] \to N &\textrm{ is endpoint-fixed homotopic to the}\\
    &\textrm{curve } s \mapsto f \big( \pi \circ \Phi_s(v) \big) \textrm{ for } 0 \leq s \leq t \}\textrm{.}
\end{align*}
Intersection turns out to be invariant under homotopy, so one may define the \textbf{intersection} of $[F]$ by $i([F]) = i(f)$ for any $f \in [F]$. If $g$ and $h$ denote the Riemannian metrics on $M$ and $N$, respectively, a \textbf{homothety} is a map $f : M \to N$ such that both $f^*(h) = cg$ for some $c \geq 0$ and the image under $f$ of each geodesic in $M$ minimizes length within its endpoint-fixed homotopy class. For each $n \in \nat$, denote by $c_n$ the volume of the unit sphere $S^n \subseteq \real^{n+1}$, where, by convention, $S^0 = \{ -1, 1 \} \subseteq \real$ has volume $c_0 = 2$.

\begin{theorem}[Croke--Fathi]\label{croke--fathi}
    Let $M$ and $N$ be Riemannian manifolds and $n = \dim(M)$. If $[F]$ is a homotopy class of maps from $M$ to $N$, then, for any $C^1$ map $f \in [F]$,
    \[
        E(f) \geq \frac{n}{2 c_{n-1}^2 \vol(M)} i^2([F])\textrm{,}
    \]
    with equality if and only if $f$ is a homothety.
\end{theorem}

\noindent There are natural generalizations of energy and length to maps into semi-Finsler manifolds. For the class of maps to which Theorem \ref{nnrc theorem} applies, the intersection is a constant multiple, depending only on the dimensions of the factors in $M_1 \times \tor^k$, of the length the totally geodesic surjection $S$.

For a more general class of NNRC-like domains, energy may be used to identify totally geodesic maps in addition to certain types of homotheties. A version of this phenomenon is recorded in Theorem \ref{main theorem}, the statement of which is rather technical. When $N$ has no conjugate points, a slightly simpler statement holds.

\begin{theorem}\label{main theorem for no conjugate points}
    Let $M$ be a compact $n$-dimensional $C^1$ Riemannian manifold and $\psi_1 : M_1 \times \tor^k \to M$ a finite covering map, where $M_1 \times \tor^k$ appears in a NNRC diagram \eqref{diagram1} that commutes isometrically and in which the manifold $M_0$ is compact. Let $N$ be a Riemannian manifold with no conjugate points and $[F]$ a homotopy class of maps from $M$ to $N$. Then, for the flat semi-Finsler manifold $K$ and the totally geodesic surjection $S : M \to K$ constructed in subsection 3.5, the following holds: For any $C^1$ map $f \in [F]$,
    \[
        E(f) \geq E(S) \geq \frac{k c_k^2 c_{n-1}}{n c_n^2 c_{k-1}^2} \frac{L^2(S)}{\vol(M)} \geq \frac{1}{c_{n-1}} \frac{L^2(S)}{\vol(M)}\textrm{.}
    \]
    Moreover, each of the following holds:

    \vspace{2pt}

    \noindent \textbf{(a)} $E(f) = E(S)$ if and only if $f$ is totally geodesic;

    \vspace{2pt}

    \noindent \textbf{(b)} $E(f) = \frac{k c_k^2 c_{n-1}}{n c_n^2 c_{k-1}^2} \frac{L^2(S)}{\vol(M)}$ if and only if $f \circ \psi_1$ is constant along each $M_1$-fiber and a homothety along each $\tor^k$-fiber;

    \vspace{2pt}

    \noindent \textbf{(c)} $E(f) = \frac{1}{c_{n-1}} \frac{L^2(S)}{\vol(M)}$ if and only if $f$ is a homothety.
\end{theorem}

\noindent The final inequality, which is simply a statement about the volumes of spheres, is included because, in principle, the corresponding geometric conditions are different. However, the inequality is strict if and only if $\inf_{f \in [F]} E(f) > 0$ and $k < n$, in which case the energy level $\frac{1}{c_{n-1}} \frac{L^2(S)}{\vol(M)}$ is unattainable within $[F]$.

\subsection*{Organization of the paper}

The second section contains background information. Subsection 2.1 defines the energy and length of a map between manifolds and, following the work of Croke, recharacterizes them using Santal\'{o}'s formula. Subsection 2.2 describes the asymptotic norm of a $\integer^k$-equivariant metric on $\integer^k$. Subsection 2.3 discusses totally geodesic maps into manifolds with no conjugate points and, in particular, uses the Poincar\'{e} recurrence theorem to characterize totally geodesic maps from manifolds with finite volume that act trivially on the fundamental group. Subsection 2.4 defines the class of domains to which Theorem \ref{nnrc theorem} generalizes, which are those that are finitely covered by a diffeomorphic product in a commutative diagram inspired by the Cheeger--Gromoll splitting theorem. Subsection 2.5 is about the beta and gamma functions and the volumes of spheres.

The main theorems are proved in the third section. Subsection 3.1 contains the most general result to be proved, the statement of which is rather long and technical. Subsection 3.2 defines the semi-Finsler torus and affine surjection associated to a homotopy class of maps. Subsection 3.3 relates the intersection of a homotopy class to the length of that affine surjection. Subsection 3.4 contains the proof of the main theorem. Subsection 3.5 specializes to the case of targets with no conjugate points.

\section{Preliminaries}

\subsection{Energy and length}

All manifolds in the paper are assumed to be $C^1$, and those satisfying curvature bounds are assumed to be $C^2$. The \textbf{energy density} of a $C^1$ map $f : M \to N$ between Riemannian manifolds $(M,g)$ and $(N,h)$ is the function $e_f = \frac{1}{2} \,\trace \!<\!\cdot,\cdot\!>_{f^{-1}(TN)}$, where $f^{-1}(TN)$ is the pull-back bundle $\coprod_{x \in M} T_{f(x)} N \to M$ and $<\!\cdot,\cdot\!>_{f^{-1}(TN)}$ is the bundle pseudo-metric obtained by pulling back $h$ via $f$. The \textbf{energy} of $f$ is $E(f) = \int_M e_f \,d \vol_M$. Croke \cite{Croke1987} observed that, as a trace, $e_f(x)$ may be computed as an average over the unit sphere $S_x M \subseteq T_x M$, endowed with its usual round metric.

\begin{lemma}[Croke]\label{energy density is a trace}
    Let $M$ and $N$ be Riemannian manifolds with $n = \dim(M)$, $f : M \rightarrow N$ a $C^1$ map, and $x \in M$. Then
    \begin{equation}\label{energy density}
        \begin{aligned}
            e_f(x) &= \frac{n}{2c_{n-1}} \int_{S_x M} \|v\|_{f^{-1}(TN)}^2 \,d\vol_{S_x M}\\
            &= \frac{n}{2c_{n-1}} \int_{S_x M} \|df(v)\|_N^2 \,d\vol_{S_x M}\textrm{.}
        \end{aligned}
    \end{equation}
\end{lemma}

\noindent The expressions on the right-hand side of \eqref{energy density} depend only on the norm on $TN$. If $N$ is endowed with only a Finsler semi-norm $\| \cdot \|_N$, then the \textbf{energy density} of a $C^1$ map $f : M \to N$ is the function on $M$ defined by \eqref{energy density}, and the \textbf{energy} of $f$ is $E(f) = \int_M e_f \,d \vol_M$. Since the Liouville measure is locally the product $\mathrm{Liou}_{SM} = \vol_M \times \vol_{S^{n-1}}$, one has (cf. \cite{Croke1987}) that
\begin{equation}\label{energy integral}
    \begin{aligned}
        E(f) &= \frac{n}{2c_{n-1}} \int_M \int_{S_x M} \|df(v)\|_N^2 \,d\vol_{S_x M} \,d\vol_M\\
        &= \frac{n}{2c_{n-1}} \int_{SM} \|df(v)\|_N^2 \,d\mathrm{Liou}_{SM}
    \end{aligned}
\end{equation}

\noindent Jost \cite{Jost1997} gave definitions of energy density and energy for maps from measure spaces into metric spaces. Centore \cite{Centore2000} showed that the above definitions agree with Jost's in the Finsler setting and yield a sensible energy functional, in that its minimizers have vanishing Laplacian.

By analogy with energy, the \textbf{length density} of a $C^1$ map $f : M \to N$ at a point $x \in M$ is defined here to be
\[
    \ell_f(x) = \sqrt{\frac{n}{2c_{n-1}}} \int_{S_x M} \|df(v)\|_N \,d\vol_{S_x M}\textrm{,}
\]
and the \textbf{length} of $f$ is $L(f) = \int_M \ell_f \,d \vol_M$, so that
\[
    L(f) = \sqrt{\frac{n}{2c_{n-1}}} \int_{SM} \|df(v)\|_N \,d\mathrm{Liou}_{SM}\textrm{.}
\]
When $n = 1$, this definition agrees with the usual length of a curve in a semi-Finsler manifold. According to the Cauchy--Schwarz inequality, $e_f \geq \frac{1}{c_{n-1}} \ell_f^2$ and
\begin{equation}\label{cauchy--schwarz}
    E(f) \geq \frac{1}{\vol(SM)} L^2(f) = \frac{1}{c_{n-1}\vol(M)} L^2(f)\textrm{,}
\end{equation}
with equality if and only if $\| df \|$ is constant.

When $M$ has boundary $\partial M \neq \emptyset$, denote by $\nu$ the inward-pointing unit normal vector field along the $C^1$ portion of $\partial M$ and $S^+ \partial M$ the inward-pointing unit vectors there. That is,
\[
    S^+ \partial M = \{ w \in S(\partial M) \st g(w,\nu) > 0 \}\textrm{.}
\]
For each $w \in S^+ \partial M$, denote by $l(w) \in (0,\infty]$ the supremum of times at which the geodesic $t \mapsto \pi \circ \Phi_t(w)$ is defined and $\varsigma_w$ the restriction of that geodesic to $[0,l(w)]$, so that $l(w) = L(\varsigma_w)$. Set $U = \{ (w,t) \in S^+ \partial M \times [0,\infty) \st t \leq l(w) \}$. Denote by $\vol_{S^+ \partial M}$ the measure obtained by restricting $\vol_{\partial M} \times \vol_{S^{n-1}}$ to $S^+ \partial M$. Santal\'{o} showed that, at $(w,t) \in U$,
\[
    \Phi|_U^*(d\mathrm{Liou}_{SM}) = g(w,\nu) \, d\vol_{S^+ \partial M} \, dt\textrm{;}
\]
Consequently, integrals over $X = \Phi(U)$ may be reformulated as integrals over $S^+ \partial M$ \cite{Santalo1952,Santalo1976}.

\begin{theorem}[Santal\'{o}'s formula]
    Let $M$ be a complete Riemannian manifold with boundary. If $f : X \to [0,\infty)$ is integrable, where $X = \Phi(U)$ as above, then
    \[
        \int_X f \,d\mathrm{Liou}_{SM} = \int_{S^+ \partial M} \big[ \int_0^{l(w)} f(\Phi_t(w)) \, dt \big] g(w,\nu) \,d\vol_{S^+ \partial M}(w) \textrm{.}
    \]
\end{theorem}

\noindent Applying Santal\'{o}'s formula to \eqref{energy integral}, Croke \cite{Croke1984,Croke1987} observed that the energy of a map is bounded below by an integral over $S^+ \partial M$, with equality if geodesics in all directions reach $\partial M$ in finite time.

\begin{lemma}[Croke]\label{energy as a santalo integral}
    Let $M$ be a complete Riemannian manifold with boundary. If $f : M \to N$ is a $C^1$ map into a Riemannian manifold, then
    \[
        E(f) \geq \frac{n}{2c_{n-1}} \int_{S^+ \partial M} E(\varsigma_w) g(w,\nu) \,d\vol_{S^+ \partial M} (w)\textrm{,}
    \]
    with equality if $\Phi|_U : U \to SM$ is surjective.
\end{lemma}

\noindent The above conclusion also holds when $N$ is a semi-Finsler manifold, and one may similarly bound the length integral.

\subsection{Asymptotic semi-norm of a periodic metric on $\integer^m$}

It was shown by Burago \cite{Burago1992} that any $\integer^m$-equivariant Riemannian metric on $\real^m$ is within finite Gromov--Hausdorff distance of a normed space. In Section 8.5 of \cite{BuragoBuragoIvanov2001}, this is partially generalized to distance functions on $\integer^m$.

\begin{theorem}[Burago--Burago--Ivanov]\label{asymptotic norm}
    Let $d : \integer^m \times \integer^m \to [0,\infty)$ be a distance function that's equivariant with respect to the action of $\integer^m$ on itself by addition. Then there exists a unique semi-norm $\| \cdot \|_\infty$ on $\real^m$ such that
    \[
        \| v \|_\infty = \lim_{n \to \infty} \frac{d(0,nv)}{n}
    \]
    for all $v \in \integer^m$. This semi-norm has the following properties:

    \vspace{2pt}

    \noindent \textbf{(i)} $\frac{d(0,v)}{\|v\|_\infty} \to 1$ uniformly as $\|v\|_\infty \to \infty$;

    \vspace{2pt}

    \noindent \textbf{(ii)} If $d$ is the orbit metric of a free and properly discontinuous action, with compact quotient, of $\integer^m$ on a length space, then $\| \cdot \|_\infty$ is a norm.
\end{theorem}

\noindent The above terminology is defined in \cite{BuragoBuragoIvanov2001}. The semi-norm $\| \cdot \|_\infty$ is the \textbf{asymptotic semi-norm} of $d$.

If $N$ is a Riemannian manifold and $G$ is a finitely generated and torsion-free Abelian subgroup of $\pi_1(N)$, then, with respect to a fixed isomorphism $G \cong \integer^m$, the action of $G$ on $N$ induces an orbit metric on $\integer^m$. Though the corresponding asymptotic semi-norm $\| \cdot \|_\infty$ depends on the choice of isomorphism, in what follows the appropriate isomorphisms will be implicitly understood. Elementary arguments show that $\| \cdot \|_\infty$ is independent of the choice of basepoint for $\pi_1(N)$. The following is a routine application of the triangle inequality.

\begin{lemma}\label{triangle inequality}
    Let $N$ be a Riemannian manifold and $G \cong \integer^m$ a finitely generated and torsion-free Abelian subgroup of $\pi_1(N)$. Then the corresponding orbit metric and asymptotic semi-norm satisfy $\|v\|_\infty \leq d(0,v)$ for all $v \in \integer^m$.
\end{lemma}

\subsection{Totally geodesic maps into manifolds with no conjugate points}

A map $f : M \to N$ between length spaces is \textbf{totally geodesic} if the composition $f \circ \gamma : (a,b) \rightarrow N$ is a geodesic whenever $f : (a,b) \rightarrow M$ is a geodesic. When $M$ and $N$ are Riemannian, an elementary argument shows that totally geodesic maps have full regularity (see \cite{Dibble2014} for details). The same doesn't necessarily hold for Finsler $N$; for instance, there are singular geodesics in $\real^2$ with respect to the constant Finsler norm $ae_1 + be_2 \mapsto |a| + |b|$.

If $N$ is Riemannian and $\gamma : [a,b] \rightarrow N$ is a geodesic, then \textbf{$\gamma(a)$ and $\gamma(b)$ are conjugate along $\gamma$} if there exists a nontrivial normal Jacobi field $J$ along $\gamma$ that vanishes at the endpoints. A Riemannian manifold $N$ has \textbf{no conjugate points} if no two points of $N$ are conjugate along any geodesic connecting them. A complete $N$ has no conjugate points exactly when each pair of points in its universal covering space $\hat{N}$ is connected by a unique minimal geodesic (see \cite{O'Sullivan1974} and \cite{GromollKlingenbergMeyer1968}). In this case, $\hat{N}$ is diffeomorphic to $\real^n$, $N$ is aspherical, and $\pi_1(N)$ is torsion-free \cite{Hurewicz1936}. According to the classical theorem of Cartan--Hadamard, manifolds with nonpositive sectional curvature have no conjugate points, as lengths of Jacobi fields are convex.

\begin{lemma}\label{totally geodesic into no conjugate points}
    Let $M$ be a complete and connected Riemannian manifold with finite volume and $N$ a length space that admits a locally isometric covering from a space $\tilde{N}$ in which every geodesic is minimal. If $f : M \to N$ is a totally geodesic map with trivial induced homomorphism on $\pi_1(M)$, then $f$ is constant.
\end{lemma}

\begin{proof}
    Assuming the result is false, one may find an open $U \subseteq TM$, contained in a set of the form $\{ v \in TM \st \pi(v) \in U_0, \| v \| \leq C_0 \}$ with compact closure, such that $df \neq 0$ on $U$. Fix $v \in U$, and write $x = \pi(v)$. Let
    \[
        C = \max \{ \| df(w) \| \st \pi(w) \in \overline{U}_0, \| w \| \leq C_0 \} > 0\textrm{,}
    \]
    $V = U \cap TB(x,\varepsilon/C)$, and $D = \min_{w \in \overline{V}} \| df(w) \| > 0$. By the Poincar\'{e} recurrence theorem, there exist $T > 2\varepsilon/D$ and $w \in V$ such that $z = \Phi_T(w) \in V$. Write $y = \pi(w)$. If $\alpha$ and $\beta$ are minimal geodesics from $x$ to $y$ and $z$, respectively, then $L(\alpha),L(\beta) < \varepsilon/C$. Since the image of the concatenation $\alpha \cdot \Phi(w,\cdot)|_{[0,T]} \cdot \beta^{-1}$ under $f$ is homotopically trivial in $N$, it lifts to $\tilde{N}$. The lifts of $f \circ \alpha$, $f \circ \Phi(w,\cdot)|_{[0,T]}$, and $f \circ \beta^{-1}$ are minimal geodesics in $\tilde{N}$ which satisfy $L \big( f \circ \Phi(w,\cdot)|_{[0,T]} \big) > 2\varepsilon$, $L(f \circ \alpha) < \varepsilon$, and $L(f \circ \beta^{-1}) < \varepsilon$. This contradicts the triangle inequality.
\end{proof}

\noindent The requirement that $M$ have finite volume cannot be dropped from Lemma \ref{totally geodesic into no conjugate points}, as shown by the universal covering map onto any complete and connected manifold with no conjugate points.

\subsection{Commutative diagrams}

This subsection describes the types of domains to which Theorem \ref{nnrc theorem} generalizes. Let $M_0$ and $M_1$ be $C^1$ manifolds, $\pi_0 : M_0 \times \real^k \to \real^k$ and $\pi_1 : M_1 \times \tor^k \to \tor^k$ projection onto the second components, and $\psi : M_0 \times \real^k \to M_1 \times \tor^k$ and $\phi : \real^k \to \tor^k$ covering maps. The \textbf{NNRC diagram}

\begin{equation}\label{diagram1}
    \begin{gathered}
        \xymatrix{ M_0 \times \real^k \ar[r]^-{\pi_0} \ar[d]^-\psi & \real^k \ar[d]^-\phi \\ M_1 \times \tor^k \ar[r]^-{\pi_1} & \tor^k}
    \end{gathered}
\end{equation}

\noindent is said to \textbf{commute isometrically} if it commutes, $M_0$, $M_0 \times \real^k$, and $M_1 \times \tor^k$ are connected Riemannian manifolds, $M_0 \times \real^k$ has a product metric with a flat $\real^k$-factor, and $\psi$ and $\phi$ are local isometries. It's emphasized that $M_1 \times \tor^k$ might not have a product metric and that $\pi_1$ is not necessarily a Riemannian submersion.

This definition is motivated by the following result of Cheeger--Gromoll \cite{CheegerGromoll1972}, a consequence of their celebrated splitting theorem.

\begin{theorem}[Cheeger--Gromoll]\label{cheeger--gromoll splitting theorem}
    Let $M$ be a compact $C^2$ Riemannian manifold with nonnegative Ricci curvature. Then $M$ is finitely and locally isometrically covered by a Riemannian manifold $M_1 \times \tor^k$ in a diagram of the form \eqref{diagram1} that commutes isometrically, in which $M_0$ is compact and simply connected.
\end{theorem}

\begin{proof}
    This is almost exactly the statement of Theorem 9 in \cite{CheegerGromoll1972}, which holds in the case of nonnegative Ricci curvature by Theorem 2 of \cite{CheegerGromoll1971}. However, it's not immediately clear that the diffeomorphism $\hat{M} \to M_1 \times \tor^k$ therein makes the diagram \eqref{diagram1} commute. One may verify that the diffeomorphism they construct is canonical and takes the image of each $M_0$-fiber to the correct $M_1$-fiber.
\end{proof}

\begin{remark}
    As in \cite{CheegerGromoll1972}, let $\integer$ act on the isometric product $S^2 \times \real$ by rotation by irrational multiples of $2\pi$ in the first component and translations in the second. Then the quotient is diffeomorphic to $S^2 \times S^1$ but not finitely covered by an isometric product. Modifying this so that the initial metric on $S^2$ is no longer round, but still admits an isometric $S^1$ action, one may obtain a diagram \eqref{diagram1} that commutes isometrically but in which $M_1 \times \tor^k$ has some negative Ricci curvature.
\end{remark}

When the diagram \eqref{diagram1} commutes, $\psi$ restricts on each $M_0$-fiber to a covering map onto an $M_1$-fiber. Moreover, there exist a covering map $\chi : M_0 \to M_1$ and a diffeomorphism $\varphi : M_0 \times \real^k \to M_0 \times \real^k$ such that the diagram

\begin{equation}\label{diagram2}
    \begin{gathered}
        \xymatrix{ M_0 \times \real^k \ar[r]^-\varphi \ar[dr]_-{\chi \times \phi} & M_0 \times \real^k \ar[r]^-{\pi_0} \ar[d]^-\psi & \real^k \ar[d]^-\phi \\ & M_1 \times \tor^k \ar[r]^-{\pi_1} & \tor^k}
    \end{gathered}
\end{equation}

\noindent commutes. The map $\chi$ may be taken as $\rho_1 \circ \psi \circ \hat{\iota}_{\hat{x}}$ for any $\hat{x} \in \real^k$, where $\hat{\iota}_{\hat{x}}$ is inclusion and $\rho_1$ is projection onto the first component of $M_1 \times \tor^k$. The map $\varphi$ is a lift of $\chi \times \phi$ along $\psi$, which may be shown to exist by general homotopy theory.

\begin{lemma}\label{bounded diameter}
    Suppose the diagram \eqref{diagram1} commutes isometrically, $M_1$ is compact, $N$ is a Riemannian manifold with Riemannian universal covering map $\pi : \hat{N} \to N$, and $f : M_1 \times \tor^k \to N$ is a continuous function such that $f \circ \iota_z(\cdot) = f(\cdot,z)$ induces the trivial homomorphism on $\pi_1(M_1)$ for each $z \in \tor^k$. Then there exist maps $\hat{f}$ and $\hat{F}$ such that diagram
    \begin{equation}\label{diagram3}
        \begin{gathered}
            \xymatrix{ M_0 \times \real^k \ar[r]^-{\hat{F}} \ar[d]^-\psi \ar[dr]^-{\hat{f}} & \hat{N} \ar[d]^-\pi \\ M_1 \times \tor^k \ar[r]^{f} & N }
        \end{gathered}
    \end{equation}
    commutes. Moreover, for any $R \geq 0$, there exists $C \geq 0$ such that any lift $\hat{F}$ satisfies $\mathrm{diam} \big( \hat{F}(M_0 \times B(\hat{z},R)) \big) \leq C$ for all $\hat{z} \in \real^k$.
\end{lemma}

\begin{proof}
    Write $\hat{f} = f \circ \psi$. The topological assumption on $f$ implies that $\hat{f}$ lifts to a map $\hat{F} : M_0 \times \real^k \to \hat{N}$ such that the diagram \eqref{diagram3} commutes. The homotopy lifting property implies that $\hat{F} \circ \varphi$ is constant on each fiber of $\chi \times \mathrm{id} : M_0 \times \real^k \to M_1 \times \real^k$. It follows that $\hat{F} \circ \varphi$ descends to a map $F : M_1 \times \real^k \to \hat{N}$ such that the diagram

    \[
        \xymatrix{ M_0 \times \real^k \ar[r]^-{\varphi} \ar[d]^-{\chi \times \mathrm{id}} & M_0 \times \real^k \ar[r]^-{\hat{F}} \ar[d]^\psi \ar[dr]^(0.6){\hat{f}} & \hat{N} \ar[d]^-\pi \\ M_1 \times \real^k \ar[r]^{\mathrm{id} \times \phi} \ar@/^0.4pc/[rru]^(0.35){F} & M_1 \times \tor^k \ar[r]^{f} & N }
    \]

    \noindent commutes.

    Fix $R \geq 0$. Define $D : \tor^k \to [0,\infty)$ by $D(z) = \mathrm{diam} \big( F(M_1 \times B(\hat{z},R)) \big)$ for any $\hat{z} \in \phi^{-1}(z)$. Since $\phi$ and $\pi$ are local isometries, $D$ is well defined. One may show that $D$ is independent of the choice of lift $\hat{F}$. By continuity, $D$ is bounded above by some $C \geq 0$. The result follows.
\end{proof}

\noindent It will also help to characterize certain totally geodesic maps from finite-volume $M_1 \times \tor^k$ into manifolds with no conjugate points. Note that, when the diagram \eqref{diagram1} commutes isometrically, each $M_1$-fiber of $M_1 \times \tor^k$ must be totally geodesic.

\begin{lemma}\label{totally geodesic from nnrc into no conjugate points}
    Suppose the diagram \eqref{diagram1} commutes isometrically, $M_1 \times \tor^k$ has finite volume, $N$ is a Riemannian manifold such that every geodesic in $\hat{N}$ is minimal, and $f : M_1 \times \tor^k \to N$ is a continuous function such that $f \circ \iota_z$ induces the trivial homomorphism on $\pi_1(M_1)$ for each $z \in \tor^k$. Then $f$ is totally geodesic if and only if $f$ is constant along each $M_1$-fiber and, for each $p \in M_1$, $f \circ \iota_p$ is totally geodesic.
\end{lemma}

\begin{proof}
    Suppose $f$ is constant along each $M_1$-fiber and, for each $p \in M_1$, $f \circ \iota_p$ is totally geodesic. Since $\phi$ is a local isometry and $\pi_1 \circ \psi = \phi \circ \pi_0$, $\hat{f}$ is constant along each $M_0$-fiber and each $\hat{f} \circ \iota_{\hat{p}}$ is totally geodesic. Since $M_0 \times \real^k$ has a product metric, $\hat{f}$ is totally geodesic, which, since $\psi$ is a local isometry, means $f$ is as well.

    Conversely, suppose $f$ is totally geodesic. Note that $f$ must be $C^1$. Since $M_1 \times \tor^k$ has finite volume and each $M_1$-fiber is totally geodesic, the coarea formula implies that almost all $M_1$-fibers have finite volume. By Lemma \ref{totally geodesic into no conjugate points} and a continuity argument, $f$ must be constant along each $M_1$-fiber. A short additional argument now shows that each $f \circ \iota_p$ is totally geodesic.
\end{proof}

\begin{lemma}\label{fundamental domains}
    Suppose the diagram \eqref{diagram1} commutes. If $\chi$ is finite with $\kappa_0$ sheets and $B \subseteq \real^k$ is a fundamental domain of $\phi$, then $M_0 \times B$ is the union of $\kappa_0$ fundamental domains of $\psi$.
\end{lemma}

\begin{proof}
    From the diagram \eqref{diagram2}, one sees that, whenever $A \subseteq M_0$ and $B \subseteq \real^k$ are fundamental domains of $\chi$ and $\phi$, respectively, $\varphi(A \times B)$ is a fundamental domain of $\psi$. For any $\hat{p} \in M_0$, $\pi_0 \circ \varphi \circ \hat{\iota}_{\hat{p}}$ is a deck transformation of $\phi$. Therefore, $B_0 = (\rho_0 \circ \varphi \circ \hat{\iota}_{\hat{p}})^{-1}(B)$ is a fundamental domain of $\phi$. For any $\hat{z} \in \real^k$, $\rho_0 \circ \varphi \circ \hat{\iota}_{\hat{z}}$ is an automorphism of $M_0$. Thus $M_0 \times B = \varphi(M_0 \times B_0)$, from which the result follows.
\end{proof}

\begin{remark}
    If the diagram \eqref{diagram1} commutes isometrically, $\Gamma$ is the deck transformation group of $\psi$, and $\mathscr{I}(M_0)$ and $\mathscr{I}(\real^k)$ are the isometry groups of $M_0$ and $\real^k$, respectively, then $\Gamma \subseteq \mathscr{I}(M_0) \times \mathscr{I}(\real^k)$.
\end{remark}

\begin{remark}
    Suppose the diagram \eqref{diagram1} commutes isometrically, $M_0$ has finite volume, and $\chi$ has $\kappa_0$ sheets. Since $\psi$ restricts on each $M_0$-fiber to a local isometry $\chi_{\hat{z}} = \rho_1 \circ \psi \circ \hat{\iota}_{\hat{z}}$ onto a totally geodesic $M_1$-fiber of $M_1 \times \tor^k$, and the number of sheets of $\chi_{\hat{z}}$ must be constant in $\hat{z} \in \real^k$, each $M_1$-fiber has volume $\frac{1}{\kappa_0} \vol(M_0)$.
\end{remark}

\subsection{Beta and gamma functions}

An inequality about the volumes of the unit spheres $S^n \subseteq \real^{n+1}$ will be used in the proof of Theorem \ref{main theorem}.

\begin{lemma}\label{volumes of spheres}
    Let $n \geq k$. Then $n c_n^2 c_{k-1}^2 \leq k c_k^2 c_{n-1}^2$, with equality if and only if $n = k$.
\end{lemma}

\noindent Let $\complex_+ = \{ z \in \complex \st \mathrm{Re}(z) > 0 \}$. Define the \textbf{beta function} $B : \complex_+ \times \complex_+ \to \complex$ by
\[
    B(x,y) = \int_0^1 t^{x-1}(1-t)^{y-1} \,dt
\]
and the \textbf{gamma function} $\Gamma : \complex_+ \to \complex$ by
\[
    \Gamma(z) = \int_0^\infty t^{z-1} e^{-t} \,dt\textrm{.}
\]
The gamma function continuously extends $(n-1)!$, as $\Gamma(1) = 1$ and $\Gamma(z+1) = z\Gamma(z)$ for all $z \in \complex_+$. These are related to each other, and to the volumes of spheres, by the following well-known equalities.

\begin{lemma}\label{beta and gamma}
    Each of the following holds:

    \vspace{2pt}

    \noindent \textbf{(a)} $B(x,y) = \frac{\Gamma(x)\Gamma(y)}{\Gamma(x+y)}$ for all $x,y \in \complex_+$;

    \vspace{2pt}

    \noindent \textbf{(b)} $c_{n-1} = \frac{2\pi^{n/2}}{\Gamma(\frac{n}{2})}$ for all $n \in \nat$;

    \vspace{2pt}

    \noindent \textbf{(c)} $B(\frac{k}{2},\frac{l}{2}) = \frac{2c_{k+l-1}}{c_{k-1} c_{l-1}}$ for all $k,l \in \nat$.
\end{lemma}

\noindent The following is an immediate consequence of Theorem 1 in \cite{BustozIsmail1986}.

\begin{theorem}[Bustoz--Ismail]\label{bustoz--ismail}
    The function $x \mapsto \sqrt{x}\frac{\Gamma(x)}{\Gamma(x + \frac{1}{2})}$ is strictly decreasing on $(0,\infty)$.
\end{theorem}

\noindent Theorem \ref{bustoz--ismail} and Lemma \ref{beta and gamma}(a) together imply the following inequality for the beta function, which has, to the best of my knowledge, gone unremarked upon in the literature.

\begin{lemma}\label{beta inequality}
    For any $x,y \geq 0$, $B(x+\frac{1}{2},y) < \sqrt{\frac{x}{x+y}} B(x,y)$.
\end{lemma}

\noindent This may be used to prove Lemma \ref{volumes of spheres}.

\begin{proof}[Proof of Lemma \ref{volumes of spheres}]
    If $n = k$, then equality is clear. If $n > k$, write $l = n - k$. Then Lemma \ref{beta inequality} implies that
    \[
        B^2 \big( \frac{k + l}{2},\frac{l}{2} \big) < \frac{k}{k+l} B^2 \big( \frac{k}{2},\frac{l}{2} \big)\textrm{,}
    \]
    which by Lemma \ref{beta and gamma}(c) is equivalent to
    \[
        \frac{4c_{k+l}^2}{c_k^2 c_{l-1}^2} < \frac{4kc_{k+l-1}^2}{(k+l)c_{k-1}^2 c_{l-1}^2}\textrm{.}
    \]
    This is equivalent to the strict form of the desired inequality.
\end{proof}

\begin{remark}\label{gurland's inequality}
    The inequality $B(\frac{k+1}{2},\frac{l}{2}) < \sqrt{\frac{k}{k+l}} B(\frac{k}{2},\frac{l}{2})$ for $k,l \in \nat$ may also be derived from Gurland's inequality $\frac{\Gamma(\frac{n+1}{2})}{\Gamma(\frac{n}{2})} < \frac{n}{\sqrt{2n+1}}$ \cite{Gurland1956}.
\end{remark}

\section{Proof of the main theorems}

\subsection{Statement of the main technical result}

Theorem \ref{main theorem for no conjugate points} is built upon the following result, which contains no assumptions on the target manifold.

\begin{theorem}\label{main theorem}
    Let $M$ be a compact $n$-dimensional $C^1$ Riemannian manifold and $\psi_1 : M_1 \times \tor^k \to M$ a finite covering map, where $M_1 \times \tor^k$ appears in a NNRC diagram \eqref{diagram1} that commutes isometrically and in which the manifold $M_0$ is compact. Let $N$ be a Riemannian manifold and $[F]$ a homotopy class of maps from $M$ to $N$ such that $F \circ \psi_1$ acts trivially on $\pi_1(M_1)$. Then, for the flat semi-Finsler torus $\tor^m$ and totally geodesic surjection $T : \tor^k \to \tor^m$ constructed in subsection 3.2, the following holds: For any $C^1$ map $f \in [F]$,
    \[
        E(f) \geq \vol(M) e_T \geq \frac{1}{c_{k-1}} \vol(M)\ell_T^2 \geq \frac{n c_n^2 c_{k-1}}{k c_k^2 c_{n-1}^2}  \vol(M)\ell_T^2\textrm{.}
    \]
    Moreover, each of the following holds:

    \vspace{2pt}

    \noindent \textbf{(a)} If $E(f) = \vol(M) e_T$, then $f$ is totally geodesic;

    \vspace{2pt}

    \noindent \textbf{(b)} $E(f) = \frac{1}{c_{k-1}} \vol(M)\ell_T^2$ if and only if $f \circ \psi_1$ is constant along each $M_1$-fiber and a homothety along each $\tor^k$-fiber;

    \vspace{2pt}

    \noindent \textbf{(c)} $E(f) = \frac{n c_n^2 c_{k-1}}{k c_k^2 c_{n-1}^2} \vol(M)\ell_T^2$ if and only if $f$ is a homothety and either $M_1$ is a point or $f$ is constant.
\end{theorem}

\noindent The next three subsections are devoted to proving this, while the last specializes to the case of targets with no conjugate points.

\subsection{The semi-Finsler torus and totally geodesic surjection associated to $[F]$}

Throughout this section, $M$ and $[F]$ will be assumed to satisfy the hypotheses of Theorem \ref{main theorem}. The flat metric on $\tor^k$ will be denoted by $g$. The constant $\kappa_1$ will denote the number of sheets of the covering map $\psi_1 : M_1 \times \tor^k \to M$, which is assumed to be a local isometry. For a fixed $f \in [F]$, let $\tilde{f} = f \circ \psi_1$ and $\hat{f} = \tilde{f} \circ \psi$. The topological assumption on $F \circ \psi_1$ is equivalent to the statement that $\tilde{f} \circ \iota_z$ induces the trivial homomorphism on $\pi_1(M_1)$ for all $z \in \tor^k$, where $\iota$ denotes inclusion.

Fix $\hat{p} \in M_0$, and let $(p,x) = \psi(\hat{p},0)$. Then
\[
    G = \tilde{f}_*(\pi_1(M_1 \times \tor^k)) = (\tilde{f} \circ \iota_x)_*(\pi_1(M_1)) (\tilde{f} \circ \iota_p)_*(\pi_1(\tor^k)) = (\tilde{f} \circ \iota_p)_*(\pi_1(\tor^k))
\]
is a free Abelian group of rank $0 \leq m \leq k$. Since the metric on $\real^k$ is flat, one may suppose that $\phi$ is the quotient map induced by a lattice $\Gamma_1$ equal to the span, with integer coefficients, of a set of vectors $V = \{ v_1,\ldots,v_k \}$. Each geodesic segment $t \mapsto tv_i$, where $t \in [0,1]$, descends via $\phi$ to a closed geodesic $s_i$ based at $x$, and $\{[s_1],\ldots,[s_k]\}$ is a minimal generating set for $\pi_1(\tor^k)$. Let $\sigma_i = \tilde{f} \circ \iota_p \circ s_i$. Without loss of generality, one may suppose that $G$ is generated by $[\sigma_1],\ldots,[\sigma_m]$.

Denote by $d_G$ the orbit metric induced by the action of $G \cong \integer^m$ on $\hat{N}$, by $d_{\integer^m}$ the induced metric on $\integer^m$, and by $\|\cdot\|_\infty$ the corresponding asymptotic semi-norm on $\real^m$. Endow $\tor^m = \real^m / \integer^m$ with the constant semi-Finsler metric $\|\cdot\|_\infty$. For each $j = 1,\ldots,m$, let $w_j$ be the vector in $\integer^m$ corresponding to $[\sigma_j]$. There exist integers $s_{ij} \in \integer$ such that $[\sigma_i] = \sum_{j=1}^m s_{ij} [\sigma_j]$ for each $i = 1,\ldots,k$. Define a linear map $\hat{T} : \real^k \to \real^m$ by $\hat{T}(v_i) = \sum_{j=1}^m s_{ij} w_j$. Then $\hat{T}$ descends to a totally geodesic surjection $T : \tor^k \to \tor^m$. Let $\tilde{S} : M_1 \times \tor^k \to \tor^m$ be the map that's constant on each $M_1$-fiber and agrees with $T$ on each $\tor^k$-fiber.

Let $\hat{F} : M_0 \times \real^k \to \hat{N}$ be a lift of $\tilde{f}$ as in Lemma \ref{bounded diameter}, and let $\hat{d}$ be the distance function on $\real^k$ obtained as the quotient of the pull-back $\hat{F}^*(d_{\hat{N}})$ by $M_0$. The following relates $\|\hat{T}(\cdot)\|_\infty$ to $\hat{d}$. The basic idea of the proof is to use Theorem \ref{asymptotic norm}(i), but one must account for $v$ that point in irrational directions.

\begin{lemma}\label{asymptotic inequalities}
     Each of the following holds:

    \vspace{2pt}

    \noindent \textbf{(a)} There exists $D \geq 0$ such that, for all $v \in T_{\hat{z}} \real^k$, $\|\hat{T}(v)\|_\infty \leq \hat{d} (\hat{z}, \hat{z} + v) + D$;

    \vspace{2pt}

    \noindent \textbf{(b)} For all $v \in T_{\hat{z}} \real^k$, $\| \hat{T}(v) \|_\infty = \lim_{t \to \infty} \frac{\hat{d}(\hat{z},\hat{z} + tv)}{t}$.
\end{lemma}

\begin{proof}
    \textbf{(a)} By Lemma \ref{bounded diameter}, there exists $D_0 \geq 0$ such that
    \[
        d_{\hat{N}}(\hat{F}(\hat{q}_0,\hat{z}_0),\hat{F}(\hat{q}_1,\hat{z}_1)) \leq D_0
    \]
    whenever $d_{\real^k}(\hat{z}_0,\hat{z}_1) \leq \mathrm{diam}(\tor^k)$. Fix $u_0,u_1 \in \Gamma_1$ such that
    \[
        d_{\real^k}(\hat{z}, u_0),d_{\real^k}(\hat{z} + v, u_1) \leq \mathrm{diam}(\tor^k)\textrm{.}
    \]
    For $i = 0,1$, let $\alpha_i : [0,1] \to M_1 \times \tor^k$ be the loop $\alpha_i(s) = (p,\phi(su_i))$, and let $\hat{\alpha}_i : [0,1] \to M_0 \times \real^k$ be the lift of $\alpha_i$ along $\psi$ satisfying $\hat{\alpha}_i(0) = (\hat{p},0)$. Then
    \begin{equation}\label{inequality1}
        |d_{\hat{N}}(\hat{F} \circ \hat{\alpha}_1(1),\hat{F} \circ \hat{\alpha}_0(1)) - \hat{d} (\hat{z}, \hat{z} + v)| \leq 2D_0\textrm{.}
    \end{equation}
    Write $u_i = \sum_{j=1}^k c_{ij} v_j$, so that $\pi_1 \circ \alpha_i \in \sum_{j=1}^k c_{ij} [s_j]$. Then $\tilde{f} \circ \alpha_i \in \sum_{j=1}^m c_{ij}[\sigma_j]$. By construction,
    \begin{equation}\label{inequality2}
    \begin{aligned}
        d_{\hat{N}}(\hat{F} \circ \hat{\alpha}_1(1),\hat{F} \circ \hat{\alpha}_0(1)) &= d_G([\tilde{f} \circ \alpha_1],[\tilde{f} \circ \alpha_0])\\
            &= d_{\integer^m} \big(\sum_{j=1}^m c_{1j} w_j, \sum_{j=1}^m c_{0j} w_j \big)\\
            &= d_{\integer^m}(0,\hat{T}(u_1 - u_0))\\
            &\geq \|\hat{T}(u_1 - u_0)\|_\infty\textrm{,}
    \end{aligned}
    \end{equation}
    where the final inequality follows from Lemma \ref{triangle inequality}.

    Write $\|\hat{T}\|_\infty = \max_{|\hat{w}| = 1} \|\hat{T}(\hat{x} + \hat{w})\|_\infty$ and $D_1 = \|\hat{T}\|_\infty \mathrm{diam}(\tor^k)$. Since $\hat{T}$ is linear,
    \begin{equation}\label{inequality3}
    \begin{aligned}
        \big| \|\hat{T}(u_1 - u_0)\|_\infty - \|\hat{T}(v)\|_\infty \big| &\leq \|\hat{T}(\hat{z} + v) - \hat{T}(u_1)\|_\infty + \|\hat{T}(\hat{z}) - \hat{T}(u_0)\|_\infty\\
            &\leq 2D_1\textrm{.}
    \end{aligned}
    \end{equation}
    The result follows from \eqref{inequality1}-\eqref{inequality3} with $D = 2D_0 + 2D_1$.

    \vspace{2pt}

    \noindent \textbf{(b)} If $v = 0$, the result is clear, so let $v \neq 0$. The argument is simplified by first proving two special cases.

    Suppose $v = \sum_{i=1}^m a_i v_i$ for $a_i \in \real$. Fix $\varepsilon > 0$. By Theorem \ref{asymptotic norm}(i), there exists $C \geq 0$ such that
    \[
        \big| d_{\integer^m}(0,w) - \|w\|_\infty \big| \leq \varepsilon \|w\|_\infty
    \]
    whenever $w \in \integer^m$ satisfies $\|w\|_\infty \geq C$. Let $t_0 = (C+2D_1)/\|v\|_\infty$. For any $t \geq t_0$, there exist $u_0,u_1 \in \Gamma_1$, as in the proof of (a), such that \eqref{inequality1}-\eqref{inequality3} hold with respect to the corresponding $\alpha_i$. By \eqref{inequality3}, $\|\hat{T}(u_1 - u_0)\|_\infty \geq C$. So
    \begin{equation}\label{inequality4}
        \big| d_{\integer^m}(0,\hat{T}(u_1 - u_0)) - \|\hat{T}(u_1 - u_0)\|_\infty \big| \leq \varepsilon \|\hat{T}(u_1 - u_0)\|_\infty\textrm{.}
    \end{equation}
    As in \eqref{inequality2}, $d_{\integer^m}(0,\hat{T}(u_1 - u_0)) = d_{\hat{N}}(\hat{F} \circ \hat{\alpha}_1(1), \hat{F} \circ \hat{\alpha}_0(1))$. This and \eqref{inequality1}, \eqref{inequality3}, and \eqref{inequality4} imply
    \[
        \big| \hat{d} (\hat{z}, \hat{z} + tv) - \|\hat{T}(tv)\|_\infty \big| \leq \varepsilon\|\hat{T}(tv)\|_\infty + 2D_0 + 2(1+\varepsilon)D_1\textrm{.}
    \]
    The conclusion follows in this case.

    The case $v = \sum_{i=m+1}^k a_i v_i$ is easier, since one may take $u_1 - u_0 = \sum_{i=m+1}^k b_i v_i$ for $b_i \in \integer$. It follows that $\hat{d} (\hat{z}, \hat{z} + tv) \leq 2D_0$ and, consequently,
    \[
        \|\hat{T}(v)\|_\infty = 0 = \lim_{t \to \infty} \frac{\hat{d} (\hat{z}, \hat{z} + tv)}{t}\textrm{.}
    \]

    The general case follows by writing $v = v_0 + v_1$, where $v_0 = \sum_{i=1}^m a_i v_i$ and $v_1 = \sum_{i=m+1}^k a_i v_i$. Then $\|\hat{T}(v)\|_\infty = \|\hat{T}(v_0)\|_\infty$. The proof is completed by applying the triangle inequality to $\hat{d}$.
\end{proof}

\subsection{Length and intersection}

Write $l = \mathrm{dim}(M_0)$, so that $n = \mathrm{dim}(M) = k + l$. As the results of this subsection will all hold trivially when $n = 0$, suppose $n > 0$, which ensures that $i([F])$ is well defined. Define constants
\[
    d_{kl} = \left\{ \begin{array}{ccc} 0 & \textrm{if} & k = 0 \\ \frac{c_{k+l}}{c_k} \sqrt{\frac{2c_{k-1}}{k}} & \textrm{if} & k > 0 \end{array} \right. \textrm{.}
\]

\begin{theorem}\label{length and intersection}
    Length and intersection are related by $i([F]) = d_{kl} \vol(M) \ell_T$.
\end{theorem}

\begin{proof}
    If $k = 0$, the result holds trivially, so suppose $k > 0$. One may show that the length of $T$ is independent of the choice of representative $f \in [F]$ used in its construction, so suppose, without loss of generality, that $f$ is $C^1$. Let $C \geq 0$ be an upper bound for $|df|$ on the unit sphere bundle $SM$. Then $\phi_t(w) \leq Ct$ for all $w \in SM$.

    Let $\hat{F} : M_0 \times \real^k \to \hat{N}$ be a lift of $\tilde{f} = f \circ \psi_1$ guaranteed by Lemma \ref{bounded diameter}. It follows from Lemma \ref{asymptotic inequalities}(b) that, for each $w \in S_x M$, $\tilde{w} \in S_{\tilde{x}}(M_1 \times \tor^k)$, and $\hat{w} = (\hat{u},\hat{v}) \in S_{\hat{x}} (M_0 \times \real^k)$ such that $\psi_*(\hat{w}) = \tilde{w}$ and $(\psi_1)_*(\tilde{w}) = w$,
    \[
        \|dT(\tilde{w})\|_\infty = \|\hat{T}(\hat{v})\|_\infty = \lim_{t \to \infty} \frac{d_{\hat{N}}(\hat{F}(\hat{p},\hat{x} + t\hat{v}),\hat{F}(\hat{p},\hat{x}))}{t} = \lim_{t \to \infty} \frac{\phi_t(w)}{t}\textrm{.}
    \]
    Thus
    \begin{equation}\label{intersection integral}
    \begin{aligned}
        i([F]) &= \int_{SM} \lim_{t \to \infty} \frac{\phi_t(w)}{t} \,d\mathrm{Liou}_{SM}\\
            &= \frac{1}{\kappa_1}\int_{M_1 \times \tor^k} \int_{S_{(\tilde{p},\tilde{x})} (M_1 \times \tor^k)} \|d\tilde{S}(\tilde{w})\|_\infty \,d\vol_{S_{(\tilde{p},\tilde{x})}}(M_1 \times \tor^k) d\mathrm{vol}_{M_1 \times \tor^k}\textrm{,}
    \end{aligned}
    \end{equation}
    where the first equality follows from the bounded convergence theorem. Note that
    \begin{equation}\label{integral over sphere bundle}
        \int_{S_{\tilde{x}} \tor^k} \|dT(\tilde{v})\|_\infty d\vol_{S_{\tilde{x}}} = \sqrt{\frac{2c_{k-1}}{k}} \ell_T\textrm{.}
    \end{equation}
    If $l = 0$, then $n = k$, $M_1 \times \tor^k \cong \tor^k$, and the result follows immediately. Suppose $l > 0$. Let $\rho : \{ (\tilde{u},\tilde{v}) \in S_{(\tilde{p},\tilde{x})} (M_1 \times \tor^k) \st \tilde{v} \neq 0 \} \to S_{\tilde{x}} \tor^k$ be defined by $\rho(\tilde{w}) = d\pi_0(\tilde{w})/|d\pi_0(\tilde{w})|$, i.e, $\rho(\tilde{u},\tilde{v}) = \tilde{v}/|\tilde{v}|$. Applying the coarea formula with respect to $\rho$ yields
    \[
    \begin{aligned}
        \int_{S_{(\tilde{p},\tilde{x})} (M_1 \times \tor^k)} \|d\tilde{S}&(\tilde{w})\|_\infty \,d\vol_{S_{(\tilde{p},\tilde{x})} (M_1 \times \tor^k)}\\
            &= \int_{S_{\tilde{x}} \tor^k} \|dT(\tilde{v})\|_\infty \int_{\rho^{-1}(\tilde{v})} |d\pi_0(\tilde{w})|^k \,d\vol_{\rho^{-1}(\tilde{v})} d\vol_{S_{\tilde{x}} \tor^k}\textrm{.}
    \end{aligned}
    \]
    Another application of the coarea formula shows that
    \[
        \int_{\rho^{-1}(\tilde{v})} |d\pi_0(\tilde{w})|^k d\vol_{\rho^{-1}(\tilde{v})} = c_{l-1} \int_0^1 r^k (1-r^2)^{\frac{l}{2} - 1} \,dr = \frac{1}{2}c_{l-1} B \big( \frac{k+1}{2},\frac{l}{2} \big)\textrm{.}
    \]
    Thus
    \begin{equation}\label{integral over product sphere bundle}
    \begin{aligned}
        \int_{S_{(\tilde{p},\tilde{x})} (M_1 \times \tor^k)} \|d\tilde{S}(\tilde{w})\|_\infty \,&d\vol_{S_{(\tilde{p},\tilde{x})} (M_1 \times \tor^k)}\\
            &= \frac{1}{2} c_{l-1} B \big( \frac{k+1}{2},\frac{l}{2} \big) \int_{S_{\tilde{x}} \tor^k} \|dT(\tilde{v})\|_\infty d\vol_{S_{\tilde{x}} \tor^k}\textrm{.}
    \end{aligned}
    \end{equation}
    The result follows from \eqref{intersection integral}-\eqref{integral over product sphere bundle} and Lemma \ref{beta and gamma}(c).
\end{proof}

\subsection{Main inequalities}

This subsection begins with a simple integral inequality, the proof of which is elementary.

\begin{lemma}\label{elementary inequality}
    Let $f : [a,b] \to \real$ be a measurable function, and let $a = x_0 < x_1 < \cdots < x_n = b$ be a partition of $[a,b]$ for some $n \geq 1$. Then
    \[
        \sum_{i=0}^{n-1} \frac{[ \int_{x_i}^{x_{i+1}} f(t) \,dt ]^2}{x_{i+1} - x_i} \geq \frac{[\int_a^b f(t) \,dt]^2}{b-a}\textrm{.}
    \]
\end{lemma}

\noindent It is now possible to prove Theorem \ref{main theorem}.

\begin{proof}[Proof of Theorem \ref{main theorem}]
    Without loss of generality, suppose that the map $f \in [F]$ used in the construction of $S$ is $C^1$. Let $\hat{F} : M_0 \times \real^k \to \hat{N}$ be a lift of $\tilde{f}$ of the form in Lemma \ref{bounded diameter}. Denote by $\kappa_0$ the number of sheets of the covering map $\chi$ in diagram \eqref{diagram3} and, as before, by $\kappa_1$ the number of sheets of $\psi_1 : M_1 \times \tor^k \to M$.

    For each $r \in \nat$, define a parallelotope $P_r = \{ \sum_{i=1}^k t_i v_i \st 0 \leq t_i \leq r \}$, where, as before, $\Gamma_1 = \{ \sum_{i=1}^k m_i v_i \st m_i \in \integer \}$. By Lemma \ref{fundamental domains}, $M_0 \times P_r$ is the union of $\kappa_0 r^k$ fundamental domains of $\psi$. Thus
    \begin{equation}\label{equation1}
    \begin{aligned}
        E(f) &= \frac{1}{\kappa_1 \kappa_0 r^k} E(\hat{F}|_{M_0 \times P_r})\\
        & \geq \frac{1}{\kappa_1 \kappa_0 r^k} \int_{M_0} E(\hat{F} \circ \iota_{\hat{q}}|_{P_r}) \,d\vol_{M_0}
    \end{aligned}
    \end{equation}
    for each $\hat{q} \in M_0$. For such $\hat{q}$ and $w \in S^+ \partial P_r$, define $\varsigma_{\hat{q},w} : [0,l(w)] \to N$ by $\varsigma_{\hat{q},w} = \hat{F} \circ \iota_{\hat{q}} \circ \gamma_w$, where $\gamma_w$ is the geodesic in $\real^k$ with initial vector $w$. By approximating $P_r$ from within by a sequence of smooth submanifolds with boundary that round off its edges, applying the condition for equality in Lemma \ref{energy as a santalo integral} to each, and taking a limit, it follows that
    \[
        E(\hat{F} \circ \iota_{\hat{q}}|_{P_r}) = \frac{k}{2c_{k-1}} \int_{S^+ \partial P_r} E(\varsigma_{\hat{q},w})g(w,\nu)\,d\vol_{S^+ \partial P_r}\textrm{.}
    \]
    Therefore,
    \begin{equation}\label{equation2}
        E(f) \geq \frac{1}{\kappa_1 \kappa_0 r^k} \int_{M_0} \frac{k}{2c_{k-1}} \int_{S^+ \partial P_r} E(\varsigma_{\hat{q},w}) g(w,\nu) \,d\vol_{S^+ \partial P_r}\textrm{,}
    \end{equation}
    with equality if and only if $\hat{F}$ is constant along each $M_0$-fiber.

    By Lemma \ref{asymptotic inequalities}(a), there exists $D \geq 0$ such that
    \[
        \frac{L^2(\varsigma_{\hat{q},w})}{2l(w)} \geq \frac{1}{2}l(w) \|\hat{T}(w)\|_\infty^2 - D\|\hat{T}(w)\|_\infty
    \]
    for all $\hat{q}$ and $w$. This and inequality \eqref{cauchy--schwarz} imply that
    \begin{equation}\label{equation3}
        E(\varsigma_{\hat{q},w}) \geq \frac{1}{2} l(w) \|\hat{T}(w)\|_\infty^2 - D\|\hat{T}(w)\|_\infty\textrm{.}
    \end{equation}
    Combining \eqref{equation2} and \eqref{equation3} yields
    \begin{equation}\label{equation4}
    \begin{aligned}
        E(f) \geq \frac{\vol(M_0)}{\kappa_1 \kappa_0 r^k} \frac{k}{2c_{k-1}} \int_{S^+ \partial P_r} &\frac{1}{2}l(w) \|\hat{T}(w)\|_\infty^2 g(w,\nu) \,d\vol_{S^+ \partial P_r}\\
            &- \frac{D\|\hat{T}\|_\infty \vol(M_0)}{\kappa_1 \kappa_0} \frac{k}{2c_{k-1}} \frac{\vol(S^+ \partial P_r)}{r^k}\textrm{.}
    \end{aligned}
    \end{equation}
    Since $\hat{T}$ is linear, $E(\hat{T}|_{P_r}) = \vol(P_r)e_{\hat{T}}$. By Lemma \ref{energy as a santalo integral}, rounding off the edges of $P_r$ as in the proof of \eqref{equation2}, one has that
    \[
        e_{\hat{T}} = \frac{E(\hat{T}|_{P_r})}{\vol(P_r)} = \frac{k}{2c_{k-1}} \frac{1}{\vol(P_r)} \int_{S^+ \partial P_r} \frac{1}{2}l(w) \|\hat{T}(w)\|_\infty^2 g(w,\nu) \,d\vol_{S^+ \partial P_r}\textrm{.}
    \]
    Since $e_T = e_{\hat{T}}$,
    \begin{equation}\label{equation5}
        \vol(M)e_T = \frac{\vol(M_0)}{\kappa_1 \kappa_0 r^k} \frac{k}{2c_{k-1}} \int_{S^+ \partial P_r} \frac{1}{2}l(w) \|\hat{T}(w)\|_\infty^2 g(w,\nu) \,d\vol_{S^+ \partial P_r}\textrm{.}
    \end{equation}
    An elementary argument shows that $\vol(S^+ \partial P_r)/r^k \to 0$ as $r \to \infty$. The inequality $E(f) \geq \vol(M)e_T$ follows from this, \eqref{equation4}, and \eqref{equation5} by letting $r \to \infty$.

    \noindent \textbf{(a)} Suppose $E(f) = \vol(M)e_T$. Fix $r \in \nat$. For each $R = (r_1,\ldots,r_k) \in \{ 0,\ldots,r-1 \}^k$, let $Q_R$ be the translation of $P_1$ by $\sum_{i=1}^k r_i v_i$. By Lemma \ref{energy as a santalo integral},
    \[
        \int_{S^+ \partial P_r} E(\varsigma_{\hat{q},w}) g(w,\nu) \,d\vol_{S^+ \partial P_r} = \sum_R E(\hat{F} \circ \iota_{\hat{q}}|_{Q_R})
    \]
    for each $\hat{q} \in M_0$. By symmetry,
    \[
        \frac{1}{r^k} \int_{M_0} \,\int_{S^+ \partial P_r} E(\varsigma_{\hat{q},w}) g(w,\nu) \,d\vol_{S^+ \partial P_r} = \int_{M_0} E(\hat{F} \circ \iota_{\hat{q}}|_{P_1}) \,d\vol_{M_0}
    \]
    for each $r \in \nat$. Combining this with the arguments used to prove \eqref{equation1}-\eqref{equation5} yields
    \begin{align*}
        \vol(M)e_T &= E(f)\\
            &\geq \frac{1}{\kappa_1 \kappa_0} \frac{k}{2c_{k-1}} \int_{M_0} \,\int_{S^+ \partial P_1} E(\varsigma_{\hat{q},w}) g(w,\nu) \,d\vol_{S^+ \partial P_1} \,d\vol_{M_0}\\
            &= \liminf_{r \to \infty} \frac{1}{\kappa_1 \kappa_0} \frac{k}{2c_{k-1}} \frac{1}{r^k} \int_{M_0} \, \int_{S^+ \partial P_r} E(\varsigma_{\hat{q},w}) g(w,\nu) \,d\vol_{S^+ \partial P_r} \,d\vol_{M_0}\\
            &\geq \vol(M)e_T \textrm{.}
    \end{align*}
    Thus all of the above are equalities. In particular, equality holds in \eqref{equation2} for $r = 1$, which means $\hat{F}$ is constant along each $M_0$-fiber.

    For each $\Omega = (\omega_1,\ldots,\omega_k) \in \{0,1\}^k$ and $r \in \nat$, denote by $\mu_\Omega : S\real^k \to \{0,1\}$ the indicator function of $S(rQ_\Omega)$, $l_\Omega : S\real^k \to [0,\infty)$ the function that takes each $w$ to the length of the line segment $rQ_\Omega \cap \gamma_w(\real)$, and $L_\Omega : S\real^k \to [0,\infty)$ the function that takes each $w$ to $L_\Omega(w) = \int_{-\infty}^\infty (\mu_\Omega \circ \gamma_w') \|\varsigma_{\hat{q},w}'\|_\infty \,dt$. Then
    \begin{align*}
        \int_{S^+ \partial P_r} \frac{L^2(\varsigma_{\hat{q},w})}{2l(w)} g(w,\nu) d\vol_{S^+ \partial P_r} &= \int_{S P_r} \frac{L_0^2(w)}{2l_0^2(w)} \,d\mathrm{Liou}_{S P_r}\\
            &= \frac{1}{2^k} \int_{S P_{2r}} \big[ \sum_{\Omega \in \{ 0,1 \}^k} \mu_\Omega(w) \frac{L_\Omega^2(w)}{2l_\Omega^2(w)} \big] \,d\mathrm{Liou}_{S P_{2r}}\\
            &= \frac{1}{2^k} \int_{S^+ \partial P_{2r}} \big[ \sum_{l_\Omega \neq 0} \frac{L_\Omega^2(w)}{2l_\Omega(w)} \big] g(w,\nu) \,d\vol_{S^+ \partial P_{2r}}\\
            &\geq \frac{1}{2^k} \int_{S^+ \partial P_{2r}} \frac{L^2(\varsigma_{\hat{q},w})}{2l(w)} g(w,\nu) \,d\vol_{S^+ \partial P_{2r}}\textrm{,}
    \end{align*}
    where the first and third equalities follow from Santal\'{o}'s formula and the inequality from Lemma \ref{elementary inequality}. Therefore,
    \begin{align*}
        \vol(M)e_T &= \frac{\vol(M_0)}{\kappa_1 \kappa_0} \frac{k}{2c_{k-1}} \int_{S^+ \partial P_1} E(\varsigma_{\hat{q},w}) g(w,\nu) \,d\vol_{S^+ \partial P_1}\\
            &\geq \frac{\vol(M_0)}{\kappa_1 \kappa_0} \frac{k}{2c_{k-1}} \int_{S^+ \partial P_1} \frac{L^2(\varsigma_{\hat{q},w})}{2l(w)} g(w,\nu) \,d\vol_{S^+ \partial P_1}\\
            &\geq \liminf_{r \to \infty} \frac{\vol(M_0)}{\kappa_1 \kappa_0} \frac{k}{2c_{k-1}} \frac{1}{2^{kr}} \int_{S^+ \partial P_{2^r}} \frac{L^2(\varsigma_{\hat{q},w})}{2l(w)} g(w,\nu) \,d\vol_{S^+ \partial P_{2^r}}\\
            &\geq \liminf_{r \to \infty} \frac{k\vol(M_0)}{2^{kr+1}\kappa_1 \kappa_0 c_{k-1}} \int_{S^+ \partial P_{2^r}} \big[ \frac{1}{2}l(w)\|\hat{T}(w)\|_\infty^2 - D\|\hat{T}(w)\|_\infty \big] g(w,\nu) \,d\vol_{S^+ \partial P_{2^r}}\\
            &= \vol(M)e_T \textrm{.}
    \end{align*}
    By the condition for equality in \eqref{cauchy--schwarz}, each $\varsigma_{\hat{q},w}$ must have constant speed. Lemma \ref{asymptotic inequalities}(a) implies that each $\varsigma_{\hat{q},w}$ has speed at least $\|\hat{T}(w)\|_\infty$. This, the above equalities, and \eqref{equation5} imply that each $\varsigma_{\hat{q},w}$ has speed $\|\hat{T}(w)\|_\infty$.

    It follows from the argument used to prove \eqref{inequality2}, and in particular Lemma \ref{triangle inequality}, that $\varsigma_{\hat{q},w}$ is a geodesic whenever $w \in S^+ \partial P_r$ satisfies $\gamma_w(s) - \gamma_w(0) \in \Gamma_1$ for some $s > 0$. The density of these rational vectors implies that each $\varsigma_{\hat{q},w}$ is a geodesic. By Lemma \ref{totally geodesic from nnrc into no conjugate points}, $f$ is totally geodesic.

    \vspace{2pt}

    \noindent \textbf{(b)} If $E(f) = \frac{1}{c_{k-1}} \vol(M)\ell_T^2$, then, by part (a) and Lemma \ref{totally geodesic from nnrc into no conjugate points}, $\tilde{f}$ is constant along each $M_1$-fiber. For each $p \in M_1$, $\tilde{f} \circ \iota_p$ has energy $\frac{1}{c_{k-1}} \vol(M)\ell_T^2$ and, by Theorem \ref{length and intersection}, intersection $\frac{2c_{k-1}}{k}\vol(\tor^k)\ell_T^2$. By Theorem \ref{croke--fathi}, $\tilde{f} \circ \iota_p$ is a homothety.

    Conversely, if $\tilde{f}$ is constant along each $M_1$-fiber and a homothety along each $\tor^k$-fiber, then it follows from Lemma \ref{asymptotic inequalities}(b) that $E(f) = \frac{1}{c_{k-1}} \vol(M)\ell_T^2$.

    \vspace{2pt}

    \noindent \textbf{(c)} If $E(f) = \frac{n c_n^2 c_{k-1}}{k c_k^2 c_{n-1}^2} \vol(M)\ell_T^2$, then $E(f) = \frac{1}{c_{k-1}} \vol(M)\ell_T^2$. By part (b), $\tilde{f}$ is constant along each $M_1$-fiber and a homothety along each $\tor^k$-fiber. If $\ell_T = 0$, then $E(f) = 0$, so $f$ is constant. If $\ell_T > 0$, then $n c_n^2 c_{k-1}^2 = k c_k^2 c_{n-1}^2$. Lemma \ref{volumes of spheres} implies that $n = k$, so $M_1$ is a point. Thus $\tilde{f}$ and, consequently, $f$ are homotheties.

    Conversely, suppose $\tilde{f}$ is a homothety. If $f$ is constant, then it's clear that $E(f) = \frac{n c_n^2 c_{k-1}}{k c_k^2 c_{n-1}^2} \vol(M)\ell_T^2$. If $M_1$ is a point, then $k = n$, and, by part (b), $E(f) = \frac{n c_n^2 c_{k-1}}{k c_k^2 c_{n-1}^2} \vol(M)\ell_T^2$.
\end{proof}

\begin{lemma}\label{semi-Riemannian norm}
    Suppose that every geodesic in $\hat{N}$ is minimal. If $[F]$ contains a totally geodesic map, then $\|\cdot\|_\infty$ is induced by a semi-inner product.
\end{lemma}

\begin{proof}
    One may suppose that the map $f$ is totally geodesic. By Lemma \ref{totally geodesic into no conjugate points}, $\hat{F}$ is constant along each $M_0$-fiber and totally geodesic along each $\real^k$-fiber. Therefore, $\hat{F}(M_0 \times \real^k)$ is a totally geodesic and flat $m$-dimensional submanifold of $\hat{N}$. Since each geodesic in $\hat{N}$ is minimal, it follows from Lemma \ref{asymptotic inequalities}(b) that $\|\cdot\|_\infty$ is induced by the pull-back of the inner product on $\real^m$ under the linear map $\hat{T}$.
\end{proof}

\begin{remark}
    If $N$ is compact and the induced homomorphism of $[F]$ is non-trivial on $\pi_1(M)$, then, by modifying the proof of Theorem 8.3.19 in \cite{BuragoBuragoIvanov2001}, one may show that $\|\cdot\|_\infty$ is a norm.
\end{remark}

\begin{remark}
    Lemma \ref{volumes of spheres}, and consequently the inequality in Remark \ref{gurland's inequality}, may be obtained without using the results of Bustoz--Ismail or Gurland by applying Theorems \ref{croke--fathi} and \ref{main theorem}{b} and Theorem \ref{length and intersection} to the projection $\tor^k \times \tor^l \to \tor^k$.
\end{remark}

\subsection{Targets with no conjugate points}

The aim of this subsection is to prove Theorem \ref{main theorem for no conjugate points}. Note that Theorem \ref{nnrc theorem} is an immediate consequence of Theorems \ref{cheeger--gromoll splitting theorem} and \ref{main theorem for no conjugate points}.

Suppose that $\pi_1(N)$ is torsion-free. Since $M_0$ is compact, the covering map $\chi : M_1 \to M_0$ in diagram \eqref{diagram2} must be finite. It follows that $\pi_1(M_1)$ is finite and, consequently, $F \circ \psi_1$ acts trivially on $\pi_1(M_1)$. Moreover, $H = f_*(\pi_1(M))$, being a torsion-free group containing $G \cong \integer^m$ as a finite-index subgroup, must be a Bieberbach group \cite{Charlap1986}. It follows that there exist a compact flat manifold $K$ and a finite covering map $\phi_1 : \tor^m \to K$. Endow $K$ with the flat semi-Finsler metric $\|\cdot\|_\infty$, so that $\phi_1$ is a local isometry. Since the universal cover of $K$ is contractible, there exists a map $M \to K$ with surjective induced homomorphism. Applying the classical theorem of Eells--Sampson \cite{EellsSampson1964}, or the results of \cite{Dibble2018c}, one finds that this map is homotopic to a totally geodesic surjection $S : M \to K$. By construction, $\tilde{S}$ is a lift of $S$ to $M_1 \times \tor^k \to \tor^m$.

It is clear that $E(S) = \vol(M)e_T$. Applying the coarea formula as in the proof of Theorem \ref{length and intersection}, one may show that
\[
    L^2(S) = \frac{n c_n^2 c_{k-1}}{k c_k^2 c_{n-1}} \vol(M)^2 \ell_T^2\textrm{.}
\]
Thus the inequalities in Theorem \ref{main theorem for no conjugate points} are equivalent to those in Theorem \ref{main theorem}. When $N$ has no conjugate points, $\pi_1(N)$ is torsion free, and the above hold.

It remains to prove that, when $N$ has no conjugate points, the converse of Theorem \ref{main theorem}(a) and the simpler version of (c) in Theorem \ref{main theorem for no conjugate points} hold. Suppose in this case that $f$ is totally geodesic. By Lemma \ref{totally geodesic from nnrc into no conjugate points}, $\tilde{f}$ is constant along each $M_1$-fiber and totally geodesic along each $\tor^k$-fiber. For each $\tilde{q} \in M_1$ and $w \in S\tor^k$, the geodesic $t \mapsto \tilde{f}(\tilde{q},\exp(tw))$ has minimal length within its endpoint-fixed homotopy class and, by Lemma \eqref{asymptotic inequalities}(b), speed $\|dT(w)\|_\infty$. A direct computation, simplified by Lemma \ref{semi-Riemannian norm}, shows that $E(f) = E(\tilde{f})/\kappa_1 = E(S)$, which proves the converse to (a). If $f$ is a homothety, then $f$ is totally geodesic and, by Lemma \ref{totally geodesic from nnrc into no conjugate points}, $\tilde{f}$ is constant along each $M_1$-fiber. If $f$ is non-constant, then, since there exists $a \geq 0$ such that $f*(h) = ah$, $M_1$ must be a point. Theorem \ref{main theorem for no conjugate points}(c) now follows from Theorem \ref{main theorem}(c).

\bibliography{bibliography}
\bibliographystyle{amsplain}

\end{document}